\documentclass[10pt]{amsart}
\usepackage{amssymb, paralist, xspace, graphicx, url, amscd, euscript, mathrsfs, stmaryrd, hyperref} %, showkeys}

\usepackage[all]{xy}
\SelectTips{cm}{}

\numberwithin{equation}{section}
1
\numberwithin{subsection}{section}
\allowdisplaybreaks[1]

\newenvironment{enumeratei} {\begin{enumerate}[\upshape (i)]} {\end{enumerate}}

\newtheorem{lem}{Lemma}[section]

\newtheorem{theorem}{Theorem}

\newtheorem{prop}[lem]{Proposition}
\newtheorem{cor}[lem]{Corollary}

\theoremstyle{definition}

\theoremstyle{remark}
\newtheorem{remark}[lem]{Remark}

\newcommand\cA{\mathcal{A}} \newcommand\cB{\mathcal{B}}   \newcommand\cL{\mathcal{L}}\newcommand\cM{\mathcal{M}}\newcommand\cO{\mathcal{O}}\newcommand\cU{\mathcal{U}}\newcommand\cX{\mathcal{X}}

\renewcommand\AA{\mathbb{A}}\newcommand\FF{\mathbb{F}}\newcommand\GG{\mathbb{G}}\newcommand\PP{\mathbb{P}}\newcommand\QQ{\mathbb{Q}}\newcommand\ZZ{\mathbb{Z}}

\newcommand\fm{\mathfrak{m}}

\newcommand\id{\mathrm{id}}

\newcommand\spec{\operatorname{Spec}}

\newcommand\arr{\ifinner\to\else\longrightarrow\fi}

\newcommand\mapsonto{\twoheadrightarrow}
\newcommand\hookarr{\hookrightarrow}

\renewcommand{\ss}{\operatorname{ss}}
\newcommand{\s}{\operatorname{s}}
\newcommand{\rrarrows}{\rightrightarrows}

\newcommand{\Mor}{\operatorname{Mor}}

\newcommand{\Pic}{\operatorname{Pic}}

\newcommand{\Proj}{\operatorname{Proj}}

\newcommand{\Sym}{\operatorname{Sym}}
\newcommand{\Aut}{\operatorname{Aut}}

\newcommand{\oh}{\cO}

\newcommand{\Spec}{\spec}
\newcommand{\tensor} {\otimes}

\newcommand{\iso}{\stackrel{\sim}{\arr}}

\newcommand{\Gr}{\operatorname{Gr}}
\newcommand{\GGr}{\GG \operatorname{r}}
\newcommand{\GL}{\operatorname{GL}}
\newcommand{\PGL}{\operatorname{PGL}}
\newcommand{\SL}{\operatorname{SL}}

\newcommand{\GIT}{\operatorname{GIT}}
\renewcommand{\bar}{\overline}
\newcommand{\dual}{\vee}
\newcommand{\bsp}{\hspace{-.05in}}

\def\T{{\text{\scriptsize T}}}

\newcommand{\epf}{\qed \vspace{+10pt}}

\begin{document}
\title{Computing invariants via slicing groupoids: 
\\{\tiny Gel'fand MacPherson, Gale and positive characteristic stable maps}}
\author[Alper]{Jarod Alper}

\address[Alper]{Department of Mathematics\\
Columbia University\\
2990 Broadway\\
New York, NY 10027\\
U.S.A.}
\email{jarod@math.columbia.edu}

\begin{abstract}
We offer a groupoid-theoretic approach to computing invariants.  We illustrate this approach by describing the Gel'fand-MacPherson correspondence and the Gale transform.  We also provide Zariski-local descriptions of the moduli space of ordered points in $\PP^1$.  We give an explicit description of the moduli space $M_0(\PP^1,2)$ over $\Spec \ZZ$.  In characteristic 2, the singularity at the totally ramified cover is isomorphic to the affine cone over the Veronese embedding $\PP^1 \arr \PP^4$.
\end{abstract}

\maketitle

\section{Introduction}
The central question in classical invariant theory is to describe the graded ring of invariants $\oplus_{n \ge 0} \Gamma(X, \oh(n))^G$ where $G$ is an algebraic group acting linearly on a projective space $\PP(V)$ and $X \subseteq \PP(V)$ is a $G$-invariant subvariety.  In this paper, we show that one can sometimes ``slice'' the groupoid $G \times X \rrarrows X$ by a subvariety $W \subseteq X$ which is suitably transverse to the generic orbit to produce a groupoid $R|_W \rrarrows W$ (not necessarily arising from a  group action) where it is easier to compute the invariants.  Specifically, suppose $X$ is normal and $g: W \to X$ is a finite type morphism such that the composition $G \times W \to G \times X \xrightarrow{\sigma} X$ is flat whose image $G \cdot W \subseteq X$ has a complement of codimension at least 2, then $R|_W := G \times X \times_{X \times X} W \times W \rrarrows W$ is a flat groupoid  and there is a canonical isomorphism $\Gamma(X, \oh(n))^G \iso \Gamma(W, g^* \oh(n))^{R|_W}$; see Section \ref{section-slicing}.  In stack-theoretic language, if $X$ is normal and $g: W \to X$ a finite type morphism such that the composition $W \xrightarrow{g} X \to [X/G]$ to the quotient stack is a flat morphism whose image has a complement of codimension at least 2, then  $R|_W = W \times_{[X/G]} W \rrarrows W$ is flat groupoid with the same invariants as $G \times X \rrarrows X$.  

In Section \ref{section-gelfand}, we illustrate the above idea by offering generalizations of the classical Gel'fand-MacPherson correspondence and Gale transform.  Propositions \ref{proposition-gm} and \ref{proposition-flag} recover and generalize results of \cite{gelfand-macpherson}, \cite{hu} and \cite{borcea}.  
In Section \ref{section-ordered-points}, we employ the technique of slicing to give Zariski-local descriptions of the moduli space of $n$ order points in $\PP^1$ as in \cite{hmsv}.  In fact, unlike the description in \cite{hmsv}, we offer a description of Zariski-neighborhoods around strictly semistable points which cover the quotient space.

In Section \ref{section-kontsevich}, we offer a global description of the Kontsevich moduli space of stable maps $M_0(\PP^1,2)$ as well as its GIT compactification.  Here $M_0(\PP^1,2)$ is the coarse moduli scheme parameterizing non-constant, degree 2 morphisms $\PP^1 \arr \PP^1$ modulo automorphisms of the source.  If the characteristic is not 2, then $M_0(\PP^1,2)$ has a simple description:  since giving a degree 2 morphism $\PP^1 \arr \PP^1$ is equivalent to giving two unordered points in $\PP^1$, $M_0(\PP^1, 2)$ is simply the affine scheme $(\PP^1 \times \PP^1\setminus \Delta) / \ZZ_2$.
In characteristic 2, it is more subtle to give an explicit description of $M_0(\PP^1,2)$ due to the totally ramified morphism $f: \PP^1 \arr \PP^1, [x,y] \mapsto [x^2, y^2]$.  In fact, the motivation of this paper was to understand the singularity in $M_0(\PP^1,2)$ at the point corresponding to $f$.  
The automorphism group scheme $\Aut(f)$ is the subgroup scheme $V(a^2-d^2,b^2,c^2) \subseteq \PGL_2$% = \{ \begin{pmatrix} a & b \\ c & d \end{pmatrix} \}$
, where $a,b,c$ and $d$ are matrix entries.  The group scheme $\Aut(f)$ is a finite, connected, non-reduced and non-linearly reductive group scheme of dimension 3 at the identity.  In particular, Luna's \'etale slice theorem does not offer an \'etale local description of $M_0(\PP^1,2)$ at this point.

  We will describe $M_0(\PP^1,2)$ globally by realizing it as the following geometric quotient:  since any degree 2 morphism can be written as $\PP^1 \arr \PP^1, [x,y] \mapsto [A_1x^2 + B_1xy+C_1y^2, A_2 x^2 + B_2 xy + C_2 y^2]$, one sees that $M_0(\PP^1,2) = U/ \PGL_2$, where if $V$ is the free $\ZZ$-module generated by $A_i, B_i, C_i$, then $U \subseteq \PP(V)$ is the invariant open locus consisting of basepoint free sections and $\PGL_2$ acts linearly on $V$ by acting on the homogenous coordinates $x$ and $y$; see Section \ref{section-kontsevich} for details.  The line bundle $\oh(1)$ on $\Proj \ZZ[A_i,B_i,C_i]$ has a unique $\PGL_2$-linearization.  The GIT quotient
$$\bar{M}_0^{\GIT}(\PP^1,2)=\Proj S \quad \text{where} \quad S=\ZZ[A_i,B_i,C_i]^{\PGL_2}$$
gives a compactification of $M_0(\PP^1, 2)$.

There are some obvious invariants over $\ZZ$:
\begin{equation}
\begin{aligned} \label{invariants}
	\Delta_1 & = B_1^2 - 4A_1C_1		& \textrm{(discriminant of $s_1$)} \\
	\Delta_2 & = B_2^2 - 4A_2C_2		& \textrm{(discriminant of $s_2$)} \\
	\Delta_{12} & = (B_1+B_2)^2 - 4(A_1+A_2)(C_1+C_2)		& \textrm{(discriminant of $s_1 + s_2$)}\\
	\Lambda &= (A_1C_2+C_1A_2)^2 - (A_1C_2+C_1A_2) (B_1B_2) + & \textrm{(vanishing of basepoint locus)}\\
				& \quad \, \, \, \, A_1C_1(B_2^2-2A_2C_2) + A_2C_2(B_1^2-2A_1C_1)	\\	
\Gamma & = B_1B_2 - 2A_1C_2-2C_1A_2
\end{aligned} \end{equation}	
with the relations
\begin{equation} \begin{aligned} \label{relations}
\Delta_{12} &= \Delta_1 + \Delta_2 + 2 \Gamma \\
 4 \Lambda &= \Gamma^2 - \Delta_1 \Delta_2
\end{aligned} \end{equation}

%The main theorem of Section \ref{section-kontsevich} paper is:

\begin{theorem}  \label{main_theorem}
The projective ring of invariants has the following explicit description:
\begin{enumeratei}
\item Over $\Spec \ZZ$, 
$$S = \ZZ[\Delta_1, \Delta_2, \Gamma, \Lambda]/(4\Lambda-\Gamma^2+\Delta_1 \Delta_2)$$
\item Over $\Spec \ZZ[\frac{1}{2}]$,  
$$S = \ZZ[\frac{1}{2}][\Delta_1, \Delta_2, \Gamma]$$
\item Over $\Spec \FF_2$, 
$$S \cong \FF_2[B_1, B_2, \Lambda]$$
\end{enumeratei}
In particular, over $\FF_2$, $\bar{M}^{\GIT}_0 (\PP^1, 2) \cong \PP(1,1,4)$ and $M_0 (\PP^1, 2) \cong \Spec k[X^4,X^3Y,X^2Y^2,XY^3,Y^4]$ is the cone over the Veronese embedding $\PP^1 \stackrel {\oh(4)}{\rightarrow} \PP^4$, where the origin corresponds to the totally ramified morphism $\PP^1 \arr \PP^1, [x,y] \mapsto [x^2,y^2]$.  
\end{theorem}

While the above theorem is rather modest, it is our belief that the technique of the proof is of interest and may be applicable in other invariant calculations.

%The method of proof is by cleverly slicing the groupoid $\PGL_2 \times \PP V =: R \rrarrows \PP(V)$ by a locally closed subscheme $W \hookarr \PP(V)$ to obtain a flat groupoid $R|_W \rrarrows W$ where the invariants $\bigoplus_k \Gamma(W, \oh(k))^{R|_W}$ are readily computable (see section \ref{section-slicing}).  The subscheme $W$ will be chosen such that $G \cdot W \subseteq \PP(V)$ is an open subscheme whose complement has codimension at least 2 which implies by Hartog's theorem that there is an isomorphism of graded rings $S =\bigoplus_k \Gamma(W, \oh(k))^{R|_W}$.  In stack-theoretic language, we choose a locally closed subscheme $W \hookarr \PP V$ such that morphism to the quotient stack $W \arr \cX = [\PP(V) / \PGL_2]$ is flat whose image has complement of codimension at least 2 and such that $S = \bigoplus_k \Gamma(\cX, \oh(k)] = \bigoplus_k \Gamma(W, \oh(k))^{R|_W}$ is computable where $R|_W = W \times_{\cX} W$.

\subsection*{Acknowledgments}  Ravi Vakil offered many valuable suggestions for this article.  I also thank Kevin Tucker.

\section{Slicing groupoids} \label{section-slicing}

\subsection{Groupoids} Let $S$ be a scheme.  An \emph{$S$-groupoid} is a pair of morphisms $s,t: R \rrarrows X$ of schemes over $S$ together with an identity section $e: X \arr R$, an inverse $i: R \arr R$ and a composition $c: R \times_{t,X,s} R \arr R$ satisfying the usual identities.  We say that an $S$-groupoid $s,t: R \rrarrows X$ is an \emph{fppf $S$-groupoid} if $s,t$ are flat and locally of finite presentation and $R \stackrel{(s,t)}{\arr} X \times_S X$ is quasi-compact and separated.  A \emph{line bundle on $X$ with $R$-action} is a line bundle $L$ on $X$ with an isomorphism $\alpha: s^* L \iso t^* L$ satisfying the cocycle condition.  We define the \emph{$\Gamma(S,\oh_S)$-module of $R$-invariant sections} as the equalizer
$$\Gamma(X,L)^R \arr \Gamma(X,L) \stackrel{\alpha \circ s^*, t^* }{\rrarrows} \Gamma(R, t^*L)$$

\subsection{Group actions}  If $G$ is a group scheme flat, finitely presented and separated over $S$ acting on a scheme $p: X \arr X$ where $\sigma: G \times_S X \arr X$ is the multiplication morphism, then $\sigma,p_2:G \times_S X \rrarrows X$ is an fppf $S$-groupoid where the identity, inverse and composition are defined in the obvious way.  A line bundle on $X$ with $G$-action (or equivalently a $G$-linearization) is the same a line bundle on $X$ with $R=G \times_S X$-action (i.e. a line bundle $L$ on $X$ with an isomorphism $\alpha: \sigma^* L \arr p_2^* L$ satisfying the cocycle condition).  

A $G$-linearization of the trivial sheaf $\oh_X$ is an isomorphism $\alpha: \oh_{G \times_S X} \iso \oh_{G \times_S X}$ with the cocycle condition being $(\sigma,p_2)^*\alpha = (\id \times \sigma)^* \alpha \circ p_23^* \alpha$.  This corresponds to a morphism $\Psi: G \times_S X \arr \GG_m$ satisfying $\psi(g \cdot h, x) = \Psi(g, h \cdot x) \cdot \Psi(h,x)$ for all $S$-schemes $T$ and $T$-valued points $g,h \in G(T)$ and $x \in X(T)$.  If $\Psi$ factors as $G \times_S X \stackrel{p_1}{\arr} G \stackrel{\chi}{\arr}\GG_m$, then $\chi$ is a character (i.e. a homomorphism $G \arr \GG_m$ of group schemes).  Conversely, any character $\chi$ gives a $G$-linearization of the trivial sheaf.  In particular, if $p_* \oh_X = \oh_S$, then $G$-linearizations correspond precisely to characters $G \arr \GG_m$.

\subsection{Slicing} \label{slicing_subsect} If $g: W \arr X$ is a morphism of schemes, define $R|_W$ as the fiber product
$$\xymatrix{
R|_W \ar[r]^{(s',t')} \ar[d]	& W \times W \ar[d] \\
R \ar[r]^{(s,t)}			& X \times X
}$$
Then $s',t': R|_W \rrarrows W$ is an $S$-groupoid where the identity, inverse and composition are defined in the obvious way.   There is a cartesian diagram
$$\xymatrix{
R|_{W} \ar[r] \ar[d]			& R \times_{s,X,g} W \ar[r] \ar[d]					& W \ar[d] \\
R \times_{t, X, g} W\ar[r] \ar[d]	& R \ar[r]^{s} \ar[d]^t		& X \ar[d] \\
W \ar[r]					& X \ar[r]								& [X/R]
}$$

If the composition $R \times_{t,X,g} W \arr R \stackrel{s}{\arr} X$ is flat, then $R|_W \rrarrows W$ is an fppf $S$-groupoid.  If $L$ is a line bundle on $X$ with $R$-action, then $g^*L$ is naturally a line bundle on $W$ with $R|_W$-action.  

If the composition $R \times_{t,X,g} W \arr R \stackrel{s}{\arr} X$ is flat and surjective, then the groupoids $R \rrarrows X$ and $R|_W \rrarrows W$ are Morita equivalent (i.e. the quotient stacks $[X/R]$ and $[W/R|_W]$ are isomorphic)  % add reference (see ????)
and the natural pullback
$g^*: \Gamma(X,L)^R \arr \Gamma(W, g^* L)^{R|_W}$
is an isomorphism.  If $X$ is normal, then by Hartogs' theorem we obtain the following useful proposition:
\begin{prop}
Suppose $R \rrarrows X$ is an fppf $S$-groupoid with $X$ normal.  Let $L$ be a line bundle on $X$ with $R$-action. If $g: W \to X$ is a morphism such that the composition $R \times_{t,X,g} W \arr R \stackrel{s}{\arr} X$ is flat and whose image has a complement of codimension at least 2, then the natural pullback
$$g^*: \Gamma(X,L)^R \arr \Gamma(W, g^* L)^{R|_W}$$
is an isomorphism. \epf
\end{prop}

\begin{remark}  If $G \times X \rrarrows X$ is the fppf $S$-groupoid arising from a group action, then slicing by $g: W \arr X$ often produces groupoids $R|_W \rrarrows W$ that do not arise from some group action.
\end{remark}

\subsection{Flatness}  We provide here a method to check when slicing a groupoid $R \rrarrows X$ by a locally closed subscheme $W \hookarr X$ produces an fppf $S$-groupoid $R|_W \rrarrows W$.  Recall the following version of the local criterion for flatness
\begin{prop} %\cite[IV.5.7]{sga1} 
Let $\phi: A \arr B$ be a flat, local homomorphism of local noetherian rings.  For  $f \in B$, the following conditions are equivalent::
\begin{enumeratei}
\item $f$ is a non-zero divisor and $B/(f)$ is flat over $A$.
\item $f \tensor 1$ is a non-zero divisor in $B \tensor_A A/\fm_A$.
\end{enumeratei}
\end{prop}

Let $s,t: R \rrarrows X$ is an fppf $S$-groupoid with $S$ noetherian.  Suppose $W=V(f) \hookarr X$ is defined by the vanishing locus of a section $f \in \Gamma(X,L)$ for a line bundle $L$ on $X$ with $R$-action (where $f$ is not necessarily $R$-invariant).  To show that the composition $t^{-1}(W) \arr R \stackrel{s}{\arr} X$ is flat above $x \in X$, one needs to show that for all $\rho \in t^{-1}(W) \subseteq R$ with $t(\rho) = x$, the local ring homomorphism
$$\oh_{X,x} \stackrel{s}{\arr} \oh_{R,\rho} \arr \oh_{R,\rho} / t^*f$$
is flat.  By the local criterion for flatness, this reduces to showing that $t^*f \tensor 1$ is a non-zero divisor in $\oh_{R,\rho} \tensor_{\oh_{X,x}}\oh_{X,x}/\fm_x$, where $\fm_x \subseteq \oh_{X,x}$ is the maximal ideal.  We conclude:
\begin{prop}  \label{flat_prop}
With the notation above, the composition $t^{-1}(W) \arr R \stackrel{s}{\arr} X$ is flat above $x \in X$ if $t^*f$ does not vanish at any associated point in $s^{-1}(x)$. \epf
\end{prop}

\subsection{Stacky interpretation} If $s,t: R \rrarrows X$ is an fppf $S$-group, then by \cite[Cor 10.6]{lmb} the quotient stack $\cX=[X/R]$ is an Artin stack (with separated and quasi-compact diagonal).  A line bundle $L$ on $X$ with $R$-action is precisely the data of a line bundle $\cL$ on $\cX$ and the $R$-invariant sections $\Gamma(X,L)^R = \Gamma(\cX,\cL)$.  Conversely, given any Artin stack $\cX$ and morphism $X \arr \cX$, the fiber product $R = X \times_{\cX} X$ with two projections $p_1,p_2: R \rrarrows X$ forms an groupoid (with the identity, inverse and composition naturally defined).  If $X \arr \cX$ is flat and locally of finite presentation, then $R  = X \times_{\cX} X \rrarrows X$ is an fppf $S$-groupoid.  If in addition $X \arr \cX$ is surjective, then $\cX \cong [X/R]$.

If $g: W \arr X$ is a morphism of schemes, then the $S$-groupoid $R|_W \rrarrows W$ obtained from slicing $X \times_{\cX} X= R \rrarrows X$ as in \ref{slicing_subsect} is the same as $R|_W = W \times_{\cX} W \rrarrows W$.

\subsection{Computing invariants}  To compute the global sections of $\Gamma(\cX,\cL)$, one may choose any flat, finitely presented and surjective morphism $p: X \arr \cX$ and compute the $R$-invariant sections $\Gamma(X,p^*\cL)$ where $R = X \times_{\cX} X$.  Furthermore, if $\cX$ is normal, then one may choose a flat, finitely presented $p: X \arr \cX$ such that image $p(U) \subseteq \cX$ has a complement of codimension at least 2.

 Suppose $G$ is a smooth group scheme acting on a normal scheme $X$ and $L$ is a $G$-linearization.  In many invariant theory problems, one wants to compute the graded ring of invariants $\bigoplus_{k \ge 0} \Gamma(X, L^k)^G$ as a subring of $\bigoplus_{k \ge 0} \Gamma(X,L^k)$.  If one has guesses for generators $X_i$ and relations $R_j$, then one may check that $\bigoplus_{k \ge 0} \Gamma(X, L^k)^G = \ZZ[X_i]/(R_j)$ after slicing by a morphism $g: W \arr X$ such that the composition $R \times_{t,X,g} W \arr R \stackrel{s}{\arr} X$ is flat with image having complement of codimension at least 2 (or, in other words, $p: W \arr [X/G]$ is flat such that $[X/G] \setminus p(W)$ has codimension at least 2).  Therefore, the invariant calculation can be simplified if one chooses the slice $W \arr X$ cleverly such that the invariant sections $\Gamma(W, \cL^k)^{R|_W}$ are easily computable.

%%%%%%%%%%%%%%%%%%%%%%%%%%%%%%%%%%%%%%%%%%%
\section{The Gel'fand-MacPherson correspondence and Gale transform} \label{section-gelfand}

We offer a generalization of the classical Gel'fand-MacPherson correspondence (\cite{gelfand-macpherson}) and Gale transform (\cite{gale}).

\subsection{Grassmanians} 
Let the base ring be the integers $\ZZ$.
Let $\GGr(k-1, n-1) = \Gr(k,n)$ be the Grassmanian of $(k-1)$-dimensional hyperplanes of $\PP^{n-1}$ ($k$-dimensional linear subspaces of $\AA^n$) for $0 < k < n$.  Since every $k$-linear subspace can be represented by a basis of $k$-vectors in $\AA^n$ and two basis differ by an element of $GL_k$, we can realize the Grassmanian as the geometric quotient $U_{k,n} / \GL_k$, where $GL_k$ acts on the set of $(k \times n)$-matrices $\AA^{kn}$ by left multiplication and $U_{k,n} \subseteq \AA^{kn}$ is the open $GL_k$-invariant subscheme consisting of matrices of full rank.  The Grassmanian $\Gr(k,n)$ is a smooth projective scheme of dimension $kn - k^2 = (n-k)k$.

The action of $\GL_n$ on $\AA^{kn}$ by right multiplication by the transpose induces an action on the quotient $\Gr(k,n) = U_{k,n} / \GL_k$, which corresponds to the action on $k$-dimensional subspaces of $\AA^n$ from $\GL_n$ acting linearly on $\AA^n$.  This realizes the quotient stack
$$[\Gr(k,n)/\GL_n] \subseteq [\AA^{kn} / (\GL_k \times \GL_n)]$$
as the open substack of $(k \times n)$-matrices of full rank where $\GL_k$ is acting on the left and $\GL_n$ is acting on the right.  The codimension of the complement is at least 2.

\subsection{Picard group} \label{grass_picard}
The Picard group of $\Gr(k,n) \cong \ZZ$ with very ample generator $L = (\bigwedge^n V)^{\dual}$ where $V \subseteq \oh^n$ is the universal rank $k$ sub-vector bundle on $\Gr(k,n)$.  The vector bundle $\oh^n$ on $\Gr(k,n)$ inherits a $\GL_n$-action for the isomorphism $\alpha: \sigma^* \oh^n \iso p_2^* \oh^n$ corresponding to the composition $\GL_n \times \Gr(k, n) \stackrel{p_1}{\arr} \GL_n \stackrel{\T}{\arr} \GL_n$ where $\T$ denotes transpose.  (This corresponds to the $\GL_n \times \GL_k$-action on the the trivial bundle $\AA^{kn} \times \AA^n$ given by $(g,h) \cdot (A,v) = (hAg^{\T}, gv)$ for $(g,h) \in \GL_n \times \GL_k$ and $(A,v) \in \AA^{kn} \times \AA^n$.)  The universal subbundle $V \subseteq \oh^n$ is $\GL_n$-invariant and therefore induced a $\GL_n$-linearization on $L = (\bigwedge^n V)^{\dual}$.  Denote $\cL$ the corresponding line bundle on the quotient stack $[\Gr(k,n) / \GL_n]$.  The Picard group of $[\Gr(k,n) / \GL_n] \cong \ZZ \langle \cL \rangle \oplus \ZZ \langle \oh^{(1)} \rangle$ where $\oh^{(1)}$ corresponds to the $\GL_n$-linearization of the structure sheaf $\oh_{\Gr(k,n)}$ given by the character $\GL_n \stackrel{\det^{-1}}{\longrightarrow} \GG_m$.

Alternatively, $\Pic([\AA^{kn} / (\GL_k \times \GL_n)]) \cong \ZZ \langle \cM_{0,1} \rangle \oplus \ZZ \langle \cM_{1,0} \rangle$ where $\cM_{i,j}$ is the $\GL_k \times \GL_n$-linearization of the trivial sheaf corresponding to the product of the characters $\det^{-i}: \GL_k \arr \GG_m$ and $\det^{-j}: \GL_n \arr \GG_m$.  The line bundles $\cM_{1,0}$ and $\cM_{0,1}$ restrict under the inclusion $[\Gr(k,n)/\GL_n] \subseteq [\AA^{kn} / (\GL_k \times \GL_n)]$ to $\cL$ and $\oh^{(1)}$, respectively.

\subsection{The correspondence}

Let $n$ be a positive integer and $0< k_1, \ldots, k_m < n$ positive integers such that $k= k_1 + \cdots + k_m>n$.  We consider the diagonal action of $\GL_n$ on $\Gr(k_1, n) \times \cdots \times \Gr(k_m, n)$ and we study the quotient stack $[\Gr(k_1, n) \times \cdots \times \Gr(k_m, n) / \GL_n]$.  By using the above representation of the grassmanian as a quotient, we see that if we consider the subgroup $H=\GL_{k_1} \times \cdots \times \GL_{k_m} \subseteq GL_{k}$ and $H$ acts via left multiplication on the set of $k \times n$-matrices $\AA^{kn}$, then
$$ [\Gr(k_1, n) \times \cdots \times \Gr(k_m, n) / \GL_n] \subseteq [\AA^{kn} / (H \times \GL_n)] $$
is the open substack consisting of blocks of $(k_i \times n)$-full rank matrices (that is, $[\Gr(k_1, n) \times \cdots \times \Gr(k_m, n) / \GL_n]  = [U_{k_1,n} \times \cdots \times U_{k_m,n} / (H \times \GL_n)]$).  The complement of this inclusion is a closed substack of codimension at least 2.

By taking the quotient of $\GL_n$ first on the open locus $U_{nk} \subset \AA^{kn}$ of full rank matrices, we have inclusions of open substacks
$$ [\Gr(k_1, n) \times \cdots \times \Gr(k_m, n) / \GL_n] \subseteq [\Gr(n,k) / H] \subseteq [\AA^{kn} / (H \times \GL_n)] $$
with the complement of each open inclusion of codimension at least 2.  Also note that $[\AA^{kn} / (H \times \GL_n)]$ is normal so that the restrictions of line bundles under these inclusions induces isomorphisms on Picard groups.  

The Picard group of $[\AA^{kn} / (H \times \GL_n)] \cong \ZZ^{m+1}$.  For each $\underline{j} = (j_1, \cdots, j_{m+1}) \in \ZZ^{m+1}$, let $\cM_{\underline{i}}$ be the line bundle on $[\AA^{kn} / (H \times \GL_n)]$ corresponding to the $\GL_{k_1} \times \cdots \times \GL_{k_m} \times \GL_n$-linearization of the structure sheaf given by the product of the characters $\det^{-j_1} \bsp : \GL_{k_1} \arr \GG_m, \ldots, \det^{-j_m}: \GL_{k_m} \arr \GG_m$ and $\det^{-j_{m+1}} \bsp : \GL_n \arr \GG_m$.  Let $L_i$ denote the ample generator of $\Pic(\Gr(k_i,n))$ with its $\GL_n$-linearization as in \ref{grass_picard} inducing $\cL_i$ on $[\Gr(k_i,n)/\GL_n]$ and $\oh^{(j_{m+1})}$ be the line bundle on $[\Gr(k_1, n) \times \cdots \times \Gr(k_m, n) / \GL_n]$ corresponding to the $\GL_n$-linearization of the structure given by the character $\det^{-j_{m+1}}\bsp : \GL_n \arr \GG_m$.  Similarly, let $\cL$ be the line bundle on $[\Gr(n,k)/H]$ corresponding to the ample generator and $\oh^{(j_1, \cdots j_m)}$ the line bundle corresponding to the $H$-linearization of the structure sheaf given by the product of $\det^{-j_1}\bsp : \GL_{k_1} \arr \GG_m, \ldots, \det^{-j_m} \bsp : \GL_{k_m} \arr \GG_m$.  Then
$$\xymatrix{
\Pic([\Gr(k_1, n) \times \cdots \times \Gr(k_m, n) / \GL_n]) 	& \Pic([\AA^{kn} / (H \times \GL_n)]) \ar[r] \ar[l]	& \Pic([\Gr(n,k) / H]) \\
\cL_1^{\tensor j_1} \boxtimes \cdots \boxtimes \cL_m^{\tensor j_m} \tensor \oh^{(j_{m+1})}  		& (j_1, \ldots, j_{m+1}) \ar[r] \ar[l]		&  \oh^{(j_1, \ldots, j_m)} \tensor \cL^{\tensor j_{m+1}} 
}$$

We have established the following general Gel'fand-MacPherson correspondence, which was also proven by Yi Hu in \cite[Theorem 4.2]{hu}:

\begin{prop} \label{proposition-gm}
 Let $n$ be a positive integer and $0< k_1, \ldots, k_m < n$ positive integers such that $k= k_1 + \cdots + k_m>n$.  The open immersion $[\Gr(k_1, n) \times \cdots \times \Gr(k_m, n) / \GL_n] \subseteq [\Gr(n,k) / H]$ of Artin stacks induces an isomorphism of Picard groups
$$\ZZ^{m+1} \cong \Pic([\Gr(n,k) / H]) \iso \Pic( [\Gr(k_1, n) \times \cdots \times \Gr(k_m, n) / \GL_n])$$
For a line bundle $\cL$ on $[\Gr(n,k) / H]$, we have an isomorphism of GIT quotients
$$\Gr(k_1, n) \times \cdots \times \Gr(k_m, n) //_{\cL} \GL_n \iso \Gr(n,k) //_{\cL} H$$
where
$$\begin{aligned} 
	\Gr(k_1, n) \times \cdots \times \Gr(k_m, n) //_{\cL} \GL_n & = \Proj \bigoplus_{d \ge 0} \Gamma([\Gr(k_1, n) \times \cdots \times \Gr(k_m, n) / \GL_n], \cL^{\tensor k}) \\
	\Gr(n,k) //_{\cL} H & = \Proj \bigoplus_{d \ge 0} \Gamma([\Gr(n,k) / H], \cL^{\tensor k})
\end{aligned}$$ \epf
\end{prop}

\subsection{The groups $\SL_n$ and $\PGL_n$}  If one considers the action of $\SL_n$ and $\PGL_n$ on $\PP^{n-1}$ instead of $\GL_n$ and the induced actions on product of Grassmanians, the Picard groups on the quotient stacks are different but of course the GIT quotients are the same.  

If $L$ is the very ample generator of the Picard group of $\Gr(k,n)$, then the inclusion $\SL_n \subseteq \GL_n$ induces a unique $\SL_n$-linearization of $L$.  There is no $\PGL_n$-linearization of $L$ but there is a unique $\PGL_n$-linearization of $L^{\tensor n}$.

Define the group $S(H \times \GL_n) \subseteq H \times \GL_n \subseteq \GL_{n+k}$ and $S(H) \subseteq H \subseteq \GL_k$ consisting of matrices of determinant 1.  There are non-canonical inclusions $H \hookarr S(H \times \GL_n)$ and $\GL_k \hookarr S(H \times \GL_n)$ such that $S(H \times \GL_n) / H \cong \SL_n$ and $S(H \times \GL_n) / \GL_k \cong S(H)$.  
The inclusion $S(H \times \GL_n) \subseteq H \times \GL_n$ induces a morphism of quotient stacks 
$$[\AA^{kn} / S(H \times \GL_n)] \arr [\AA^{kn} / (H \times \GL_n)]$$ and a surjection on Picard groups
$$\xymatrix{
\Pic([\AA^{kn} / (H \times \GL_n)]) \ar@{->>}[r] \ar[d]^{\wr} 		& \Pic([\AA^{kn} / S(H \times \GL_n)]) \ar[d]^{\wr} \\
\ZZ^{m+1} \ar[r]			& \ZZ^{m+1} / (j_1, \ldots, j_m, j_1 + \cdots + j_m) \cong \ZZ^m
}$$
We have inclusions of open substacks
$$ [\Gr(k_1, n) \times \cdots \times \Gr(k_m, n) / \SL_n] \subseteq [\Gr(n,k) / S(H)] \subseteq [\AA^{kn} / S(H \times \GL_n)] $$
which for a line bundle $\cL \in [\Gr(n,k) / S(H)]$ induces an isomorphism of GIT quotients
$$\Gr(k_1, n) \times \cdots \times \Gr(k_m, n) //_{\cL} \SL_n \iso \Gr(n,k) //_{\cL} S(H)$$

Similarly, define the group $P(H \times \GL_n) = (H \times \GL_n)/\GG_m$ and $P(H) = H / \GG_m$ where $\GG_m$ is embedded diagonally.  The surjection $H \times \GL_n \arr P(H \times \GL_n)$ induces a rigidification morphism of quotient stacks
$$[\AA^{kn} / (H \times \GL_n)] \arr [\AA^{kn} / P(H \times \GL_n)]$$ 
and an injection on Picard groups
$$\xymatrix{
\Pic([\AA^{kn} / P(H \times \GL_n)]) \ar[r] \ar[d]^{\wr} 		& \Pic([\AA^{kn} / (H \times \GL_n)]) \ar[d]^{\wr} \\
\ZZ^{m} \cong \{(j_1, \ldots, j_{m+1} \, | \, j_1 k_1 + \cdots + j_m k_m + j_{m+1} n = 0 \} \ar@{^(->}[r]			&  \ZZ^{m+1}
}$$
We have inclusions of open substacks
$$ [\Gr(k_1, n) \times \cdots \times \Gr(k_m, n) / \PGL_n] \subseteq [\Gr(n,k) / S(H)] \subseteq [\AA^{kn} / P(H \times \GL_n)] $$
which for a line bundle $\cL \in [\Gr(n,k) / S(H)]$ induces an isomorphism of GIT quotients
$$\Gr(k_1, n) \times \cdots \times \Gr(k_m, n) //_{\cL} \PGL_n \iso \Gr(n,k) //_{\cL} P(H)$$

%\subsection{Special case $k_i = 1$}

The most classical Gel'fand-MacPherson correspondence is in the case that each $k_i = 1$.

\begin{cor} Let $m>n$ be positive integers. There is an open immersion of Artin stacks $[(\PP^{n-1})^m/ \GL_n] \subseteq [\Gr(n,m) / \GG_m^m]$ which induces an isomorphism of Picard groups
$$\Pic([\Gr(n,m) / \GG_m^m]) \iso \Pic( [(\PP^{n-1})^m/ \GL_n])$$
For a line bundle $\cL$ on $[\Gr(n,k) / H]$, there is an isomorphism of GIT quotients
$$(\PP^{n-1})^m //_{\cL} \GL_n \iso \Gr(n,k) //_{\cL} \GG_m^m$$ \epf
\end{cor}

\subsection{A generalization to flag varieties} \label{subsection-flag}

We can generalize the main result of \cite{borcea} which offered birational equivalences between certain flag associations.   

Let $0 < d_1 < d_2 < \cdots d_k=d < n$ be integers and $\underline{d} = (d_1, \ldots d_k)$.  The flag variety $F(\underline{d}, n)$ parameterizes flags $0 \subseteq V_1 \subseteq \cdots \subseteq V_k \subseteq \AA^n$ with $\dim V_i = d_i$.   

Let $U_{dn} \subseteq \AA^{dn}$ be the open subscheme consisting of $d \times n$ matrices where for each $1 \le i \le k$, the first $d_i$ rows are full rank.  The group $\GL_d$ acts on $\AA^{dn}$ by left multiplication.  Fix a representation of a flag in $U_{dn}$ (for instance, $\begin{pmatrix} \id_d & 0 \end{pmatrix}$).  The stabilizer $P \subseteq \GL_d$ is a parabolic subgroup and  $F(\underline{d}, n)$ is the geometric quotient $U/P$.  If $0 \subseteq V_1 \subseteq \cdots V_k \subseteq \oh^n$ is the universal flag on $F(\underline{d}, n)$, then $L = (\bigwedge^d V_k)^{\dual}$ is a very ample line bundle.  The group $\GL_n$ acts on $F(\underline{d},n)$ and $L$ has a $\GL_n$-linearization as above.

Let $\underline{d}_i = (d_{i1}, \ldots, d_{il_r})$ with $0 < d_{i1} < \cdots < d_{il_i} =d < n$.  Let $\underline{e} = (e_1, \ldots, e_s)$ with $0 < e_1 < \cdots < e_s < n$ be integers.  Set $d = d_1 + \cdots + d_r$. Let $P_i \subseteq \GL_{d_i}$ be the parabolic subgroup fixing a $d_i \times n$-matrix representing some flag in $F(\underline{d}_i,n)$ and $Q \subseteq \GL_n$ be the parabolic subgroup fixing a $n \times d$-matrix representing some flag in $F(\underline{e},d)$.  

The diagonal action of $P=P_1 \times \cdots \times P_r \subseteq \GL_d$ on $F(\underline{e}, d)$ and the action of $Q \subseteq \GL_n$ on the product $F(\underline{d}_1,n) \times \cdots F(\underline{d}_r,n)$ induces
$$\xymatrix{
					& [\AA^{dn} / (P \times Q)] \\
[F(\underline{d}_1,n) \times \cdots F(\underline{d}_r) / Q] \ar@{^(->}[ur] \ar@{-->}[rr]^{\Psi}		&	&
[F(\underline{e}, d) / P] \ar@{_(->}[ul]
}$$
with $\Psi$ a birational morphism which is an isomorphism over an open substack having complement of codimension at least 2.  We conclude:
\begin{prop} \label{proposition-flag}
With the above notation, there is an isomorphism of Picard groups
$$\Pic([F(\underline{e}, d) / P] ) \iso \Pic([F(\underline{d}_1,n) \times \cdots F(\underline{d}_r, n) / Q])$$
For a line bundle $\cL$ on $[F(\underline{d}_1,n) \times \cdots F(\underline{d}_r, n) / Q]$, there is an isomorphism of GIT quotients
$$F(\underline{d}_1,n) \times \cdots F(\underline{d}_r, n) //_{\cL} \, Q \iso  F(\underline{e}, d) //_{\Psi^* \cL} \, P$$
\epf
\end{prop}

%\subsection{The Gale transform}By combining the Gel'fand-MacPherson correspondence with the duality between subbundles and quotient bundles, one can establish the Gale transform.  We first recall:

\subsection{Duality between subbundles and quotient bundles} \label{duality_sub}
There is an obvious isomorphism $\Gr(k, n) \iso \Gr(n-k,n)$ given functorially on a scheme $T$ by
$$\begin{aligned}
\Gr(k,n)(T)	& \arr \Gr(n-k,n) (T) \\
(V \subseteq \oh^n) &\mapsto ( (\oh^n/V)^\dual \subseteq \oh^n)
\end{aligned}$$

\begin{remark}
There is no apparent morphism of groupoids $(\GL_k \times U_{k,n} \rrarrows U_{k,n}) \arr (\GL_{n-k} \times U_{k,n-k} \rrarrows U_{k,n-k})$ induced from morphisms $U_{k,n} \arr U_{n-k,n}$ and $\GL_k \times U_{k,n} \arr \GL_{n-k} \times U_{k,n-k}$.
However, we can consider a bigger presentation of $\Gr(k,n) \cong \Gr(n-k,n)$ incorporating both representation.  The group $\GL_k \times \GL_{n-k}$ acts freely on $Y = \{ (A,B) \in U_{k,n} \times U_{n-k,n} \, | \, A B^{\T} = 0 \}$ such that $\Gr(k,n) \cong [Y/ (\GL_k \times \GL_{n-k})] \cong \Gr(n-k,n)$.
\end{remark}

\subsection{The Gale transform}
By combining the Gel'fand-MacPherson correspondence (Proposition \ref{proposition-gm}) with the duality between subbundles and quotient bundles, one can establish the Gale transform.

 \begin{prop} Let $n$ be a positive integer and $0< k_1, \ldots, k_m < n$ positive integers such that $k= k_1 + \cdots + k_m>n$.  Consider the diagram
$$\xymatrix{
 [\Gr(k_1, n) \times \cdots \times \Gr(k_m, n) / \GL_n] \ar@{^(->}[r] \ar@{-->}[d]^{\Phi}	& [\Gr(n,k) / H] \ar[d]^{\wr} \\
 [\Gr(k_1, k-n) \times \cdots \times \Gr(k_m, k -n)/ \GL_{k-n}] \ar@{^(->}[r]  	&  [\Gr(k-n,k) / H]
}$$
The \emph{Gale transform} $\Phi$ is a birational morphism which is an isomorphism in codimension 1 and induces an isomorphism
$$\Pic([\Gr(k_1, k-n) \times \cdots \times \Gr(k_m,k-n) / \GL_{k-n}]) \iso \Pic( [\Gr(k_1, n) \times \cdots \times \Gr(k_m, n) / \GL_n])$$
For a line bundle $\cL$ on $[\Gr(k_1, k-n) \times \cdots \times \Gr(k_m, k -n) / \GL_{k-n}]$, there is an isomorphism of GIT quotients
$$\Gr(k_1, n) \times \cdots \times \Gr(k_m, n) //_{\cL} \, \GL_n \arr \Gr(k_1, k-n) \times \cdots \times \Gr(k_m, k -n) //_{\Phi^* \cL} \, \GL_{k-n}$$ \epf
\end{prop}

\begin{remark}
Similarly one can write down a Gale transform for correspondences of products of flag varieties.  
\end{remark}

In the special case when each $k_i =1$, we recover:
\begin{cor}
Let $n<m$ be a positive integers.  Consider the diagram
$$\xymatrix{
 [(\PP^{n-1})^m) / \GL_n] \ar@{^(->}[r] \ar@{-->}[d]^{\Phi}			& [\Gr(n,m) / \GG_m^m] \ar[d]^{\wr} \\
 [(\PP^{m-n-1})^m/ \GL_{m-n}] \ar@{^(->}[r]  	&  [\Gr(m-n,m) / \GG_m^m]
}$$
The \emph{Gale transform} $\Phi$ is a birational morphism which is an isomorphism in codimension 1 which induces an isomorphism of GIT quotients
$$(\PP^{n-1})^m //_{\cL} \SL_n \iso (\PP^{m-n-1})^m //_{\cM} \SL_{m-n}$$
where $\cL = \oh(j_1) \boxtimes \cdots \oh(j_m)$ on $(\PP^{n-1})^m$ (resp. $\cL= \oh(j_1) \boxtimes \cdots \oh(j_m)$ on $(\PP^{m-n-1})^m$) with the natural $\SL_n$ (resp. $\SL_{m-n}$) linearization. 
\epf
\end{cor}

%%%%%%%%%%%%%%%%%%%%%%%%%%%%%%%%%%%%%%%%%%%%%%%%
\section{Zariski-local description of the moduli space of ordered points in $\PP^1$} \label{section-ordered-points}

The purpose of this subsection is show how slicing can be used to produce two explicit Zariski-local of descriptions of GIT quotients in the case of ordered point in $\PP^1$.  The first Zariski-local description for the quotient simply reproduces the result \cite[Lemma 4.3]{hmsv} albeit in a slightly different language.  The second Zariski-local description has the advantage that they cover the GIT quotient.

\subsection{Setup}
We consider the symmetric case for the action of $\SL_2$ on $n$ ordered points in $\PP^1$ (that is, the unique $\SL_2$-linearization of $\oh(1) \boxtimes \cdots \boxtimes \oh(1))$.  The non-symmetric case can be reduced to the symmetric case (see \cite[Theorem 1.2]{hmsv}).  The GIT quotient is the projective scheme 

$$M_n = \Proj R \quad, \text{where} \quad R = \bigoplus_d \Gamma( (\PP^1)^n, \oh(d) \boxtimes \cdots \boxtimes \oh(d) )^{\SL_2}$$

The first fundamental theorem of invariant theory says that $R$ is generated by the invariants $X_{\Delta}$ where $\Delta$ is a directed graph of vertices labeled 1 through n which each vertex having degree d and
$$X_{\Delta} = \prod_{\text{edges} \, (ij)\in\Delta}  (x_iy_j - x_jy_i) \in \Gamma( (\PP^1)^n, \oh(d) \boxtimes \cdots \boxtimes \oh(d))^{\SL_2}$$
Kempe's theorem states that $R$ is generated by invariants of lowest degree (i.e. invariants of degree 1 if $n$ is even and degree 2 if $n$ is odd).  Let $I$ be the ideal of relations (that is, $I = \ker (\ZZ[X_{\Delta}] \arr R)$ where $\Delta$ varies over degree 1 invariants). The result \cite[Theorem 1.2]{hmsv} established that $I$ is generated by relations of degree 4.  The same authors more recently have shown that in general $I$ is generated by the sign, Pl\"ucker relations and simple binomial relations when $n \neq 6$ over $\ZZ[1/12!]$.  This gives a complete and beautiful description of the projective ring of invariants $R$.

\subsection{Zariski-local descriptions}

If $n$ is odd, there are no strictly semistable points and $\PGL_2$-acts freely on $(\PP^1)^{n, \s}$ as every stable configuration has at least three distinct points..  A Zariski-local description of the quotient is easy to give.  For instance, around the point $(0,1, \infty, p_4, \ldots, p_n)$, the quotient can be described locally as $W \cap (\PP^1)^{n,\ss}$ where $W=\{ (0, 1, \infty, q_4, \ldots, q_n) \} \subseteq (\PP^1)^n$.

If $n=2m$ is even, the situation is more subtle as the point $(0, \ldots, 0, \infty, \ldots, \infty)$
is strictly semistable.  We will give two Zariski-local descriptions of the GIT quotient around this point.

\subsection{First description}
Consider the $\SL_2$-invariant open affine subscheme 
$$U = \{ P_i \ne P_j \textrm{ for } 1 \leq i \leq m < j \leq n \} \subseteq (\PP^1)^n$$
which is the non-vanishing locus of the section
$$s= \prod_{i \leq m < j \leq n} (x_iy_j - x_j y_i) \in \Gamma( (\PP^1)^n, \oh(m) \boxtimes \cdots \boxtimes \oh(m))^{\SL_2}$$
By slicing the groupoid $\SL_2 \times U \rrarrows U$ by the closed subscheme $W=\{P_1=0, P_n = \infty\} \hookarr U$, we obtain the groupoid $R|_W \rrarrows W$ where $R|_W \subseteq \SL_2 \times U$ is defined by the vanishing of $(x_1, \sigma^*x_1, y_n, \sigma^* y_n)$.  Since $\sigma^*x_1 = ax_1+by_1$ and $\sigma^*y_n = c x_n + d y_2$, $R|_W$ is defined by $x_1=y_n=b=c=0$ so that $R|_W = \GG_m \times W$.    Using Proposition \ref{flat_prop}, one sees that $\SL_2 \times W \stackrel{\sigma}{\arr} U$ is faithfully flat.  One sees that the sliced groupoid $R|_W \rrarrows W$ is the same as the groupoid induced from the action of $\GG_m \subseteq \SL_2$ on $W$.  In other words, if $\cU = [U/\SL_2]$, the composition $W \hookarr U \arr \cU$ is faithfully flat and an $\SL_2$-torsor; that is, $\cU \cong [W / \GG_m]$.  

We can write $W = \Spec \ZZ[x_i, y_j]_{x_i y_j -1}$ where $\GG_m = \Spec \ZZ[t]_t$ acts via $x_i \mapsto t x_i$, $y_j \mapsto t^{-1} y_j$.  The invariants of this action are clear:  if we set $W_{ij} = x_i y_j$, then 
$$\Gamma(W, \oh_W)^{\GG_m} =  \ZZ[ W_{ij}]_{W_{ij} -1} / (W_{ij} W_{kl} - W_{il} W_{kj} )$$. 

This reproves:

\begin{prop} \cite[Lemma 4.3]{hmsv}  \label{first_desc}
If $n = 2m$, $R = \bigoplus_d \Gamma( (\PP^1)^n, \oh(d) \boxtimes \cdots \boxtimes \oh(d) )^{\SL_2}$ and $s = \prod_{i \leq m < j \leq n} (x_iy_j - x_j y_i) \in R_m$, then there is an isomorphism
$$R_{(s)} = \ZZ[ W_{ij}]_{W_{ij} -1} / (W_{ij} W_{kl} - W_{il} W_{kj} )$$
In particular, $W// \GG_m \subseteq (\PP^1)^m //_L \SL_2$ is isomorphic to the space of $(m-1) \times (m-1)$ matrices of rank at most 1 where each entry differs from 1 with the point $(0, \cdots, 0, \infty, \cdots, \infty)$ corresponding the zero matrix. \epf
\end{prop}

\begin{remark}
While this gives a satisfying Zariski-local description of the singularity, unfortunately the open sets $U \subseteq (\PP^1)^{n,\ss}$ defined by varying the choice of a partition $\{1, \ldots n\}$ into two equal length subsets do not cover the semistable locus.  For example, if $n=6$, the point $(0,0,1,1,\infty, \infty)$ is not in any such open subset.   
\end{remark}

%%%%%%%%%%
\subsection{Second description}  The $\SL_2$-invariant open affine subscheme
$$U = \{ (P_1, \ldots, P_n) \, | \, P_1 \neq P_2, \ldots, P_{n-1} \neq P_n\}$$
is the non-vanishing locus of the section
$$s = \prod_{ 1 \le i \le m} x_{2i-1} y_{2i} - x_{2i}y_{2i-1} \in \Gamma( (\PP^1)^n, \oh(1) \boxtimes \cdots \boxtimes \oh(1))^{\SL_2}$$
This is the invariant section corresponding to the graph $\Delta = 12 \cdot 34 \cdot \cdots \cdot (n-1)n$.

Let $W \hookarr U$ be the closed subscheme defined by $P_1 = 0, P_2 = \infty$.  The composition $W \hookarr U \arr [U/\SL_2]$ is faithfully flat giving the the quotient stack representation $[U/\SL_2] \cong [W / \GG_m]$. 

Consider the Pl\"ucker embedding
$$\PP^1 \times \PP^1 \arr \PP^3, \quad ([x_i,y_i], [x_{i+1}, y_{i+1}]) \mapsto [x_i x_{i+1} \, , \, x_i y_{i+1} \, , \, y_i x_{i+1} \, , \, y_i y_{i+1}]$$
and set $\tilde A_i = x_i x_{i+1}, \tilde B_i = x_i y_{i+1}, \tilde C_i = y_i x_{i+1}, \tilde D_i = y_i y_{i+1}$.  By inverting $\tilde B_i - \tilde C_i$, we have
$$(\PP^1 \times \PP^1) \setminus \{P_i = P_{i+1} \} = \Spec \ZZ[A_i, B_i, C_i, D_i] / (A_i D_i - B_i C_i, B_i - C_i - 1)$$
 where $A_i = \tilde A_i / (\tilde B_i - \tilde C_i), \ldots, D_i = \tilde D_i / (\tilde B_i - \tilde C_i)$.  This gives the description
$$W = \Spec \ZZ[A_i, B_i, C_i,  D_i \, | \, i=3,5, \cdots n-1] /  A_i  D_i -  B_i  C_i,  B_i -  C_i - 1)$$
such that $\GG_m$ acts via $ A_i \mapsto t^2  A_i,  B_i \mapsto  B_i,   C_i \mapsto  C_i,$ and $ D_i \mapsto t^{-2}  D_i$.

\begin{prop} \label{second_desc}
If $n = 2m$, $R = \bigoplus_d \Gamma( (\PP^1)^n, \oh(d) \boxtimes \cdots \boxtimes \oh(d) )^{\SL_2}$ and $s = \prod_{ 1 \le i \le m} x_{2i-1} y_{2i} - x_{2i}y_{2i-1} \in R_1$, then there is an isomorphism
$$R_{(s)} =  \ZZ[ B_i,  C_i,  F_{ij} \, | \, i,j = 3, 5, \cdots, n-1]/ ( F_{ii} -  B_i  C_i,  B_i -  C_i - 1,  F_{ij}  F_{kl} -  F_{il}  F_{kj})$$
where $F_{ij}$ is the invariant $ A_i  D_j$.  \epf
\end{prop}

\subsection{Scheme-theoretic description of $(\PP^1)^n // \SL_2$}
The first main theorem of \cite{hmsv} in the case of the symmetric case states that over $\ZZ[1/3]$ for $n \neq 6$ the $(\PP^1)^n // \SL_2$ is scheme-theoretically cut out by the sign, Pl\"ucker, and simple binomial relations.  Their proof uses two ingredients: (1) an induction argument to prove the result away from the strictly semistable points and (2) the explicit description of Proposition \ref{first_desc} (\cite[Lemma 4.3]{hmsv}) to prove the result in a neighborhood around a strictly semistable point.

It is conceivable that Proposition \ref{second_desc} can give a direct proof of this theorem since such Zariski opens cover the GIT quotient.  In fact, let $J$ be the ideal generated by the sign, Pl\"ucker, and simple binomial relations and $S=\ZZ[X_{\Delta}]/J$ so that there is a surjective morphism $\pi: S \mapsonto R = \ZZ[X_{\Delta}]/I$ where $I$ is the ideal of relations.  Let $f= \prod_{ 1 \le i \le m} x_{2i-1} y_{2i} - x_{2i}y_{2i-1}$.  One must show that for $n \neq 6$, $\pi_{(f)}: S_{(f)} \arr R_{(f)}$.  The second description above gives an isomorphism $\gamma: R_{(f)} \arr \Gamma([W/\GG_m])$ so that there is a diagram
$$\xymatrix{
(\ZZ[X_{\Gamma}] / J)_{(f)} = S_{(f)} \ar[r]^{\pi_{(f)}} \ar[rd]^{\alpha}	&  R_{(f)} = (\ZZ[X_{\Gamma}]/I)_{(f)} \ar[d]^{\gamma} \\
								& \Gamma([W/\GG_m])  \ar@{-->}@/^1.5pc/[ul]_{\beta}
}$$
One can construct explicitly the inverse $\gamma^{-1}: \Gamma([W/\GG_m]) \arr R_{(f)}$.  It remains only to check that $\gamma^{-1}$ lifts to a morphism $\beta: \Gamma([W/\GG_m]) \arr S_{(f)}$ such that $\beta \circ \alpha = \id$.  Unfortunately, we have not been able to show this.

We stress though that Howard, Millson, Snowden and Vakil can prove the much stronger result that the ideal of relations is generated by the sign, Pl\"ucker, and simple binomial relations for $n \neq 6$.  However, it is possible that our methods could be applicable to configurations of points in $\PP^m$ for $m > 1$ where very little is currently known.

%%%%%%%%%%%%%%%%%%%%%%%%%%%%%%%%%%%%%%%%%%%%%%%%

\section{The Kontsevich moduli space of stable maps $M_0(\PP^1,2)$}  \label{section-kontsevich}

We recall from \cite{kontsevich}, \cite{fulton-pandharipande} and \cite{abramovich-oort} that the moduli stack of maps from smooth curves $\cM_0(\PP^1, 2)$ over $\Spec \ZZ$ is the category fibered in groupoids where an object over a scheme $T$ is a pair $(p: C \arr T, f: C \arr \PP^1)$ where $C \arr T$ is a smooth, proper morphism whose geometric fibers are connected, genus 0 curves and $f: C \arr \PP^1$ is a morphism such that for any geometric point $t \in T$, $f_t: C_t \arr \PP^1$ is non-constant.  A morphism $(p: C \arr T, f: C \arr \PP^1) \arr (p': C' \arr T', f: C' \arr \PP^1)$ in $\cM_0(\PP^1, 2)(T)$ is a cartesian square
$$\xymatrix{
C \ar[r]^{\alpha}	 \ar[d]		& C' \ar[d] \\
T \ar[r]^{\beta}				& T'
}$$
such that $f = f' \circ \alpha$.  It is well known that $\cM_0(\PP^1, 2)$ is an Artin stack with finite inertia. 

 Let $V=\{(A_1x^2 + B_1xy+C_1y^2, A_2 x^2, B_2 xy + C_2 y^2) \}^{\dual}$ be the dual of the free $\ZZ$-module consisting of two sections $s_1,s_2 \in \Gamma(\PP^1, \oh(2))$.  Let $X=\PP V$ and $U$ be the open subscheme of $X$ consisting of basepoint free sections.  The locus $\Delta = X \setminus U$ consisting of pairs of sections with basepoints is given by the vanishing of a degree 4 homogeneous function $\Lambda$ (see Equation \ref{invariants}).
 
$X$ has a natural action of $G=PGL_2$ and $\oh(1)$ has a unique $PGL_2$-linearization which on global sections is given by 
$$\begin{aligned}
\Gamma(X, \oh(1)) 	& \rightarrow \Gamma(G \times X, \oh \boxtimes \oh(1))\\
A_i				& \mapsto \frac{1}{ad-bc} (A_i a^2 + B_i ac + C_i c^2) \\
B_i				& \mapsto \frac{1}{ad-bc} (2A_i ab + B_i(ad+bc) + 2C_icd)  \quad ( = B_i  \, \text { in characteristic $2$)} \\
C_i 				& \mapsto \frac{1}{ad-bc} (A_i b^2 + B_i bd + C_i d^2),
\end{aligned}$$
where we have given $PGL_2$ the projective coordinates $a,b,c,d$.

\begin{prop} There is an isomorphism of Artin stacks $\cM_0(\PP^1, 2) \cong [U / PGL_2]$ over $\Spec \ZZ$.  Over $\Spec \ZZ[1/2]$, $\cM_0(\PP^1,2)$ is a separated Deligne-Mumford stack.  In characteristic 2, the stabilizer of the totally ramified degree 2 cover $\PP^1 \arr \PP^1, [x,y] \mapsto [x^2,y^2]$ is the group scheme $V(a^2-d^2,b^2,c^2) \subseteq \PGL_2$ which is not reduced and not linearly reductive.  In particular, $\cM_0(\PP^1,2)$ is a non-tame Artin stack over $\FF_2$ with finite inertia.
\end{prop}
\begin{proof}
Let $\cA$ be the prestack over the big \'etale site of affine schemes whose objects over an affine scheme $\Spec A$ are surjective $A$-module homomorphisms $\alpha: A \tensor_{\ZZ} V^{\dual} \mapsonto A$ such that for each prime $p \subseteq A$ the $k(p)$-vectors $(\alpha_p(A_1), \alpha_p(B_1), \alpha_p(C_1))$ and  $(\alpha_p(A_2), \alpha_p(B_2), \alpha_p(C_2))$ are linearly independent where $\alpha_p = \alpha \tensor k(p)$ (i.e. the induced morphism $\Spec A \arr \PP V$ factors through $U$).  The group of morphisms $\Mor(A_1 \tensor_{\ZZ} V^{\dual} \stackrel{\alpha_1}{\arr} A_1,A_2 \tensor_{\ZZ} V^{\dual} \stackrel{\alpha_1}{\arr} A_2)$ over $f: \Spec A_1 \arr \Spec A_2$ is the subgroup of elements $g \in \PGL_2(A_1)$ such that there exists $a \in A_1^*$ inducing a commutative diagram
$$\xymatrix{
 A_1 \tensor_{\ZZ} V^{\dual} \ar[r]^{\alpha_1} \ar[d]^g	& A_1 \ar[d]^a\\
 A_1 \tensor_{\ZZ} V^{\dual} \ar[r]^{\alpha_2 \tensor_{A_2} {A_1}}	& A_1
 }$$

 Let $\cB$ be the prestack whose objects over $\Spec A$ are morphisms $f: \PP^1_{A} \arr \PP^1_{A}$ such that $f^*\oh(1) \cong \oh(2)$.  The morphisms $\Mor(f_1,f_2)$ over $\Spec A_1 \arr \Spec A_2$ is the set of elements $g \in \PGL_2(A_1)$ such that there is a commutative diagram
 $$\xymatrix{
 \PP^1_{A_1} \ar[r]^{f_1} \ar[d]^g	& \PP^1_{A_1} \\
 \PP^1_{A_1} \ar[ru]_{f_2 \tensor A_1}
 }$$
 
 The stackification of $\cA$ and $\cB$ is isomorphic to $[U/\PGL_2]$ and $\cM_0(\PP^1,2)$, respectively.  There is an equivalence of categories $\cA \arr \cB$  sending $\alpha:A \tensor_{\ZZ} V^{\dual} \mapsonto A$ to the morphism $f:\PP^1_{A} \arr \PP^1_{A}$ determined by the sections $(\alpha(A_1) X^2 + \alpha(B_1)XY + \alpha(C_1) Y^2, (\alpha(A_2) X^2 + \alpha(B_2)XY + \alpha(C_2) Y^2) \in \Gamma(\PP^1_{A},\oh(2))$ 
Stackification therefore induces an isomorphism of stacks $[U/\PGL_2] \arr \cM_0(\PP^1,2)$.

For the final statements, it is well-known that $\cM_0(\PP^1,2)$ is separated.  It is easy to see that the stabilizer $G_p$ of the point $p \in \cM_0(\PP^1,2)(\FF_2)$ is defined by $V(a^2-d^2,b^2,c^2) \subseteq \PGL_2(\FF_2)$ which is clearly non-reduced.  Since the subgroup scheme of $G_p$ defined by $c=0, a=d=1$ is isomorphic to the non-linearly reductive group scheme $\alpha_2$, it follows that $G_p$ is not linearly reductive.
 \end{proof}

Over $\Spec \ZZ$, by the Keel-Mori theorem (\cite{keel-mori}), there exists a coarse moduli space $\cM_0(\PP^1, 2) \arr M_0(\PP^1, 2)$.  The scheme $M_0(\PP^1, 2)$ has the GIT compactification:
$$\bar M_0^{\GIT}(\PP^1, 2) = \PP V //_{\oh(1)} \PGL_2 := \Proj S \quad \text{where } \quad S =\bigoplus_{k \ge 0} \Gamma(\PP V, \oh(k))^{\PGL_2}$$
Then $M_0(\PP^1,2) \subseteq \bar M_0^{\GIT}(\PP^1,2)$ is open locus where $\Lambda$ doesn't vanish, where $\Lambda$ is the invariant defined in (\ref{invariants}) whose vanishing determines the basepoint locus.

%%%%%%%%%%%%%%%%%%%%%%%
\subsection{Proof of Theorem \ref{main_theorem} over $\Spec \ZZ$}

We prove that over $\ZZ$, the ring of invariants $S \cong \ZZ[\Delta_1, \Delta_2, \Gamma, \Lambda]/(4\Lambda+\Delta_1 \Delta_2-\Gamma^2)$:

\begin{proof}
Let $W \hookarr X = \PP V$ be the locally closed subscheme defined by 
$$W = \{A_1=C_2 = 0 \} \setminus \big( \{B_1=C_1 = 0\} \cup \{A_2 = B_2 = 0 \} \big)$$
We claim that $W \arr [\PP V / \PGL_2]$ is flat such that its image $[W/R|_{W}] \subseteq [\PP(V) / \PGL_2]$ has complement consisting of pairs of sections $(s_1,s_2)$ where either $s_1=0$ or $s_2=0$.  In particular, its image has complement of codimension at least 2.  To prove this, we will slice first by $A_1=0$ and then by $C_2=0$ using the local criterion for flatness to verify that both slices produce fppf groupoids.   

First, let $W_1 = \{A_1 = 0 \} \setminus \{B_1=C_1 = 0\}$ and consider the diagram
$$\xymatrix{
R|_{W_1} \ar[r] \ar[d]			& G \times W_1 \ar[r] \ar[d]					& W_1 \ar[d] \\
\sigma^{-1}(W_1) \ar[r] \ar[d]	& G \times X \ar[r]^{p_2} \ar[d]^\sigma		& X \ar[d] \\
W_1 \ar[r]					& X \ar[r]								& [X/R]
}$$
 By Proposition \ref{flat_prop}, to check that $\sigma^{-1}(W_1) \stackrel{p_2}{\arr} \PP V$ is flat, we need to check that $\sigma^* A_1 = \frac{1}{ad-bc} (A_1 a^2 + B_1 ac + C_1 c^2)$ does not vanish at any associated point in $p_2^{-1}(w) = \PGL_2 \times_{\ZZ} k(w)$ for $w \in W_1 \subseteq \PP V$.  This statement is true as $w \notin V(B_1, C_1)$.   It follows that $R|_{W_1} \arr W_1$ is an fppf groupoid such that $[W_1 / R|_{W_1}]$ is an open substack of $[\PP V / \PGL_2]$. 

Consider the groupoid $R|_W \rrarrows W$.
%Second, let $W$ be defined as above.  Consider the diagram
%$$\xymatrix{
%R|_{W_2}	\ar[rr] \ar[d]			&							& G \times W_2 \ar[r] \ar[d]					& W_2 \ar[d] \\
%p_2:\sigma^{-1}(W_2) \cap p_2^{-1}(W_1) \ar[r] \ar[d]		& R|_{W_1} \ar[r] \ar[d]			& G \times W_1 \ar[r] \ar[d]					& W_1 \ar[d] \\
%\sigma^{-1}(W_2) \ar[r] \ar[d]	& \sigma^{-1}(W_1) \ar[r] \ar[d]		& G \times X \ar[r]^{p_2} \ar[d]^\sigma		& X \ar[d] \\
%W_2 \ar[r]						& W_1 \ar[r]					& X \ar[r]								& [X/R]
% t}$$
We need to check that $\sigma^* C_2$ does not vanish at any associated point in $V(A_1,\sigma^* A_1) \subseteq \PGL_2 \times_{\ZZ} k(w)$ for $w \in W $.  Since $w \notin V(B_1,C_1)$, we see that $V(A_1,\sigma^* A_1) = V(A_1, c(B_1a + C_1 c))$ has associated points $(c)$ and $(B_1 a + C_1 c)$.  Since $w \notin V(A_2,B_2)$, $\sigma^*C_2 = \frac{1}{ad-bc} (A_2 b^2 + B_2 bd + C_2d^2)$ does not vanish at either associated point.  
It follows that $W \arr [\PP V / \PGL_2]$ is flat.  

The orbit $\PGL_2 \cdot W $ consists of pairs of sections $(s_1,s_2)$ where both $s_1$ and $s_2$ are non-zero.   Clearly $[\PP V / \PGL_2 ] \setminus [W / R|_W]$ has codimension 3.  Therefore, we may compute the invariants as the equalizer of:
$$\xymatrix{
 \bigoplus_{k=0}^{\infty} \Gamma(W,\oh(k)) \ar@<0.5ex>[r]^{\sigma} \ar@<-0.5ex>[r]_{p_2}  \ar@2{-}[d]
 	 & \bigoplus_{k=0}^{\infty} \Gamma(R|_W, \oh(k))  \ar@2{-}[d]\\
 k[B_1,C_1, A_2, B_2]  \ar@<0.5ex>[r] \ar@<-0.5ex>[r] 
 	& \Gamma(\PGL_2,\oh_{\PGL_2}) [B_1,C_1,A_2,B_2] / (c(B_1a+C_1c), b(A_2b+B_2d)) 
}$$
$$\begin{array}{ll}
	B_1 & \stackrel{\sigma^*} {\mapsto} \frac{1}{ad-bc}(B_1(ad+bc)+2C_1cd) \\
	C_1 &  \stackrel{\sigma^*} {\mapsto} \frac{1}{ad-bc}(B_1bd+C_1d^2) \\
	A_2 &  \stackrel{\sigma^*} {\mapsto} \frac{1}{ad-bc}(A_2a^2+B_2ac) \\
	B_2 &  \stackrel{\sigma^*} {\mapsto} \frac{1}{ad-bc}(2A_2ab + B_2(ad+bc)),
\end{array}$$

This computation still seems unmanageable but we will restrict the computation to each of the four components of $R|_W$:  $R_1 = V(b,c)$, $R_2 = V(c, A_2b+B_2d)$, $R_3 = V(b, B_1a+C_1c)$, and $R_4 = V(B_1a+C_1c,A_2b+B_2d)$.   Recall that if $D$ is a ring with $c,d \in D$ and $d$ a non-zero divisor in $D/(c)$ then $(c) \cap (d) = (cd)$ and $D/(cd) \hookrightarrow D/(c) \times D/(d)$ is injective.  Therefore, we have inclusions
$$\bigoplus \Gamma(R|_W, \oh(k)) \hookarr \bigoplus \Gamma(R_1, \oh(k)) \times \bigoplus \Gamma(R_2, \oh(k)) \times \bigoplus \Gamma(R_3, \oh(k)) \times \bigoplus \Gamma(R_3, \oh(k))$$
If $S_i$ is the equalizer of
$\oplus_{k}  \Gamma(W,\oh(k)) \stackrel{\sigma,p_2}{\rrarrows} \oplus_k \Gamma(R_i, \oh(k))$, then $S = \cap S_i$ is the ring of invariants.  We can now compute:
{\small
$$
\begin{array}{ll}
\begin{array}{l} i=1 \\ (b=c=0) \end{array}	& \qquad
\begin{array}{rcl}
		 \ZZ[B_1,C_1,A_2,B_2] & \stackrel{\sigma, p_2}{\rrarrows} & (\ZZ[a,d]_{ad})_0 \tensor \ZZ[B_1,C_1,A_2,B_2]\\
		 B_1				& \stackrel{\sigma}{\mapsto} & B_1 \\
		 C_1				& \stackrel{\sigma}{\mapsto} & \frac{d}{a}C_1 \\
		 A_2				& \stackrel{\sigma}{\mapsto} & \frac{a}{d}A_2 \\
		 B_2				& \stackrel{\sigma}{\mapsto} & B_2 \\
\end{array} \\
		& \implies \fbox{$S_1 = \ZZ[B_1,B_2,C_1 A_2]$}\\
%%%%i=4
\begin{array}{l} i=4 \\ (A_2b+B_2d=0, \\B_1a+C_1c=0) \end{array}	& \qquad
\begin{array}{rcl}
		 \ZZ[B_1,C_1,A_2,B_2] & \stackrel{\sigma, p_2}{\rrarrows} & (\ZZ[a,b,c,d]_{ad-bc})_0 \tensor \ZZ[B_1,C_1,A_2,B_2]/ \\
		 	&&(A_2b+B_2d=0, B_1a+C_1c=0) \\
		 B_1				& \stackrel{\sigma}{\mapsto} & -B_1\\
		 C_1				& \stackrel{\sigma}{\mapsto} &  \frac{1}{ad-bc}(B_1bd+C_1d^2)  \\
		 A_2				& \stackrel{\sigma}{\mapsto} & \frac{1}{ad-bc}(A_2a^2+B_2ac) \\
		 B_2				& \stackrel{\sigma}{\mapsto} & -B_2 \\
		 C_1A_2			& \stackrel{\sigma}{\mapsto} & \frac{1}{(ad-bc)^2} (B_1A_2a^2bcd+C_1A_2a^2d^2+B_1B_2abcd+C_1B_2acd^2) \\
		 		&	& = C_1A_2\\
\end{array} \\
		& \implies \fbox{$S_1 \cap S_4 = \ZZ[B_1^2,B_2^2, B_1B_2,C_1A_2] / ( (B_1^2)(B_2^2) - (B_1 B_2)^2 )$}
\end{array}
$$\newline	
%%%i=2
$$\begin{array}{ll}	
\begin{array}{l} i=2 \\ (c=A_2b+B_2d=0) \end{array}		& \qquad
\begin{array}{rcl}
		 \ZZ[B_1,C_1,A_2,B_2] & \stackrel{\sigma, p_2}{\rrarrows} & (\ZZ[a,b,d]_{ad})_0 \tensor \ZZ[B_1,C_1,A_2,B_2]/ (A_2b+B_2d)\\
		 B_1^2			& \stackrel{\sigma}{\mapsto} & B_1 \\
		 C_1				& \stackrel{\sigma}{\mapsto} & \frac{b}{a}B_1 + \frac{d}{a}C_1 \\
		 A_2				& \stackrel{\sigma}{\mapsto} & \frac{a}{d}A_2 \\
		 B_2				& \stackrel{\sigma}{\mapsto} & 2\frac{b}{d}A_2+B_2 = -B_2 \\
		 B_1 B_2 		& \stackrel{\sigma}{\mapsto} & -B_1 B_2 \\
		 C_1A_2				& \stackrel{\sigma}{\mapsto} & \frac{b}{d}B_1A_2+C_1A_2 = C_1A_2 - B_1B_2 \\
		 C_1A_2	-B_1B_2			& \stackrel{\sigma}{\mapsto} & C_1 A_2 - B_1 B_2\\
\end{array} \\
		& \implies \fbox{$S_1 \cap S_2 \cap S_4 = \ZZ<B_1^2,B_2^2,2C_1A_2-B_1B_2,C_1A_2(C_1A_2 - B_1B_2)>$}\\
%%%i=3
\begin{array}{l} i=3 \\ (b=B_1a+C_1c=0) \end{array}	& \qquad
\begin{array}{rcl}
		 \ZZ[B_1,C_1,A_2,B_2] & \stackrel{\sigma, p_2}{\rrarrows} & (\ZZ[a,c,d]_{ad})_0 \tensor \ZZ[B_1,C_1,A_2,B_2]/(B_1a+C_1c)\\
		 B_1				& \stackrel{\sigma}{\mapsto} & B_1+2\frac{c}{a}C_1 = -B_1 \\
		 C_1				& \stackrel{\sigma}{\mapsto} & \frac{d}{a}C_1 \\
		 A_2				& \stackrel{\sigma}{\mapsto} & \frac{a}{d}A_2+\frac{c}{d}B_2 \\
		 B_2				& \stackrel{\sigma}{\mapsto} & B_2 \\
 		 B_1 B_2 		& \stackrel{\sigma}{\mapsto} & -B_1 B_2 \\
		 C_1A_2				& \stackrel{\sigma}{\mapsto} & \frac{c}{a}B_2C_1+C_1A_2 = C_1A_2 - B_1B_2 \\
		 C_1A_2	-B_1B_2			& \stackrel{\sigma}{\mapsto} & C_1 A_2 - B_1 B_2\\
\end{array} \\
		& \implies \fbox{$S = \ZZ<B_1^2,B_2^2,2C_1A_2-B_1B_2,C_1A_2(C_1A_2 - B_1B_2) >$}\\
\end{array}$$
}

The calculation above of the invariants for $S_1$ and $S_1 \cap S_2$ is obvious.  It is also easy to see that $S_1 \cap S_2 \cap S_4$ is generated by $B_1^2$, $B_2^2$, $2C_1A_2 - B_1 B_2$ and $C_1A_2(C_1A_2-B_1B_2)$.  For instance, $S_1 \cap S_2 \cap S_4$ is the $\ZZ_2$-invariants of the ring $\ZZ[A,B,C,D]/((C-D)^2 - AB)$ where the action is trivial on $A$ and $B$ but swaps $C$ and $D$ where $A=B_1^2$, $B= B_2^2$, $C = A_1C_2$ and $D = A_1C_2 - B_1B_2$.  It follows that $S_1 \cap S_2 \cap S_4 = \ZZ[A,B,C+D, CD] / ((C-D)^2 - AB)$.  Since $(C-D)^2 - AB = (C+D)^2 - 4CD - AB$, we see that the ideal of relations is generated by $(2C_1A_2-B_1B_2)^2  - 4 ( C_1 A_2 (C_1A_2 - B_1B_2))  -  (B_1)^2 (B_2)^2$.  It follows that
$$\begin{aligned} S  &= \ZZ[B_1^2, B_2^2, 2C_1A_2 - B_1 B_2, C_1A_2(C_1A_2-B_1B_2)] /  \\
& ((2C_1A_2-B_1B_2)^2  - 4 ( C_1 A_2 (C_1A_2 - B_1B_2))  -  (B_1)^2 (B_2)^2) 
\end{aligned}$$

The invariants $\Delta_1, \Delta_2, \Gamma$ and $\Lambda $ restrict to $W$ as:
$ \Delta_1|_{W} = B_1^2$,  
 $\Delta_2|_{W}	= B_2^2$,
$ \Gamma|_{W}	 = B_1B_2 - 2C_1A_2$, and
$ \Lambda|_{W} = C_1A_2(C_1A_2-B_1B_2)$
 and we have the relation
 $$	4 \Lambda|_{W} = \Gamma|_W^2 - \Delta_1|_W \Delta_2|_W$$
\end{proof}
 
 \subsection{Proof of Theorem \ref{main_theorem} over $\Spec \ZZ[1/2]$}
 Since good GIT quotients are stable under flat base change, it follows that over $\Spec \ZZ[1/2]$, the ring of invariants is
 $$\ZZ[\frac{1}{2}][\Delta_1, \Delta_2, \Gamma]$$
 
 %%%%%%%%%%%%%%%%%%%%%
 \subsection{Proof of Theorem \ref{main_theorem} over $\Spec \FF_2$}
 
 Since $\PGL_2 \arr \Spec \ZZ$ is not linearly reductive, the ring of invariants over $\FF_2$ is not necessarily 
$$\ZZ[\Delta_1, \Delta_2, \Gamma, \Lambda]/(4\Lambda-\Gamma^2+\Delta_1 \Delta_2) \tensor \FF_2 = \FF_2[\Delta_1, \Delta_2, \Lambda]$$
and indeed there are invariants (eg. $B_1$ and $B_2$) over $\FF_2$ that do not lift to invariants over $\ZZ$. 
In characteristic 2, %the dual action is given by
%$$\begin{aligned}
%\Gamma(X, \oh(1)) 	& \rightarrow \Gamma(G \times X, \oh \boxtimes \oh(1))\\
%A_i				& \mapsto \frac{1}{ad-bc} (A_i a^2 + B_i ac + C_i c^2) \\
%B_i				& \mapsto B_i\\
%C_i 				& \mapsto \frac{1}{ad-bc} (A_i b^2 + B_i bd + C_i d^2)
%\end{aligned}$$
we immediately see that $B_1, B_2$ are invariants.  The locus $V(B_i)$ consists of sections $(s_1,s_2)$ where $s_i$ has a double root.  We see that $\Delta_1 = B_1^2$, $\Delta_2 = B_2^2$, $\Delta_{12} = B_1^2+B_2^2$ and $\Gamma = B_1B_2$ are clearly generated by $B_1$ and $B_2$.  We have

$$\Lambda = (A_1C_2+C_1A_2)^2 + (A_1C_2+C_1A_2) (B_1B_2) + A_1C_1B_2^2 + A_2C_2 B_1^2$$
 and the relation $4 \Lambda = \Gamma^2 - \Delta_1 \Delta_2$ which turns into the obvious relation $\Gamma^2 = \Delta_1 \Delta_2$. 
We repeat the above slicing argument in characteristic 2:
 
\begin{proof}
Let $W \hookarr  \PP V$ be the locally closed subscheme defined by 
$$W = \{A_1=C_2 = 0 \} \setminus \big( \{B_1=C_1 = 0\} \cup \{A_2 = B_2 = 0 \} \big)$$
By the same argument as above, $W \arr [\PP V / \PGL_2]$ is flat and $[W/R|_W] \subseteq [ \PP V/G]$ has complement of codimension 3.  As before, we need to compute the equalizer of

$$\xymatrix{
 \bigoplus_{k=0}^{\infty} \Gamma(W,\oh(k)) \ar@<0.5ex>[r]^{\sigma} \ar@<-0.5ex>[r]_{p_2}  \ar@2{-}[d]
 	 & \bigoplus_{k=0}^{\infty} \Gamma(R|_W, \oh(k))  \ar@2{-}[d]\\
 \FF_2[B_1,C_1, A_2, B_2]  \ar@<0.5ex>[r] \ar@<-0.5ex>[r] 
 	& \Gamma(G,\oh_G) [B_1,C_1,A_2,B_2] / (c(B_1a+C_1c), b(A_2b+B_2d)) 
}$$

We restrict the computation to each of the four components of $R|_W$:  $R_1 = V(b,c)$, $R_2 = V(c, A_2b+B_2d)$, $R_3 = V(b, B_1a+C_1c)$, and $R_4 = V(B_1a+C_1c,A_2b+B_2d)$:
{\small
$$\begin{array}{ll}
\begin{array}{l} i=1 \\ (b=c=0) \end{array}	& \qquad
\begin{array}{rcl}
		 \FF_2[B_1,C_1,A_2,B_2] & \stackrel{\sigma, p_2}{\rrarrows} & (\FF_2[a,d]_{ad})_0 \tensor \FF_2[B_1,C_1,A_2,B_2]\\
		 B_1				& \stackrel{\sigma}{\mapsto} & B_1 \\
		 C_1				& \stackrel{\sigma}{\mapsto} & \frac{d}{a}C_1 \\
		 A_2				& \stackrel{\sigma}{\mapsto} & \frac{a}{d}A_2 \\
		 B_2				& \stackrel{\sigma}{\mapsto} & B_2 \\
\end{array} \\
		& \implies \fbox{$S_1 = \FF_2[B_1,B_2,C_1 A_2]$}
		\end{array}
$$
$$
\begin{array}{ll}
%%%i=2
\begin{array}{l} i=2 \\ (c=A_2b+B_2d=0) \end{array}		& \qquad
\begin{array}{rcl}
		 \FF_2[B_1,C_1,A_2,B_2] & \stackrel{\sigma, p_2}{\rrarrows} & (\FF_2[a,b,d]_{ad})_0 \tensor \FF_2[B_1,C_1,A_2,B_2]/ (A_2b+B_2d)\\
		 B_1				& \stackrel{\sigma}{\mapsto} & B_1 \\
		 B_2				& \stackrel{\sigma}{\mapsto} & B_2 \\
		 C_1A_2				& \stackrel{\sigma}{\mapsto} & \frac{b}{d}B_1A_2+C_1A_2 = C_1A_2 + B_1B_2 \\
 		 C_1A_2 + B_1 B_2			& \stackrel{\sigma}{\mapsto} &  C_1A_2\\

		C_1A_2(C_1 A_2 + B_1B_2)			& \stackrel{\sigma}{\mapsto} & C_1A_2(C_1 A_2 + B_1B_2)	
\end{array} \\
		& \implies \fbox{$S_1 \cap S_2 = \FF_2[B_1,B_2,C_1A_2(C_1 A_2 + B_1B_2)	]$}\\
%%%i=3
\begin{array}{l} i=3 \\ (b=B_1a+C_1c=0) \end{array}	& \qquad
\begin{array}{rcl}
		 \FF_2[B_1,C_1,A_2,B_2] & \stackrel{\sigma, p_2}{\rrarrows} & (\FF_2[a,c,d]_{ad})_0 \tensor \FF_2[B_1,C_1,A_2,B_2]/(B_1a+C_1c)\\
		 B_1				& \stackrel{\sigma}{\mapsto} & B_1 \\
		 B_2				& \stackrel{\sigma}{\mapsto} & B_2 \\
		 C_1A_2				& \stackrel{\sigma}{\mapsto} & \frac{c}{a}B_2C_1+C_1A_2 = C_1A_2 + B_1B_2 \\
		 		 C_1A_2(C_1 A_2 + B_1B_2)			& \stackrel{\sigma}{\mapsto} & C_1A_2(C_1 A_2 + B_1B_2)	
\end{array} \\
		& \implies \fbox{$S_1 \cap S_2 \cap S_3 = \FF_2[B_1^2,B_2,2C_1A_2-B_1B_2]$}\\
%%%%i=4
\begin{array}{l} i=4 \\ (A_2b+B_2d=0, \\B_1a+C_1c=0) \end{array}	& \qquad
\begin{array}{rcl}
		 \FF_2[B_1,C_1,A_2,B_2] & \stackrel{\sigma, p_2}{\rrarrows} & (\FF_2[a,b,c,d]_{ad-bc})_0 \tensor \FF_2[B_1,C_1,A_2,B_2]/ \\
		 	&& (A_2b+B_2d=0, B_1a+C_1c=0) \\
		 B_1				& \stackrel{\sigma}{\mapsto} & B_1\\
		 C_1				& \stackrel{\sigma}{\mapsto} &  \frac{1}{ad-bc}(B_1bd+C_1d^2)  \\
		 A_2				& \stackrel{\sigma}{\mapsto} & \frac{1}{ad-bc}(A_2a^2+B_2ac) \\
		 B_2				& \stackrel{\sigma}{\mapsto} & B_2 \\
		 C_1A_2			& \stackrel{\sigma}{\mapsto} & \frac{1}{(ad-bc)^2} (B_1A_2a^2bcd+C_1A_2a^2d^2+B_1B_2abcd+C_1B_2acd^2) \\
		 		&	& = C_1A_2\\
\end{array} \\
		& \implies \fbox{$S = \FF_2[B_1,B_2,C_1 A_2]$}\\
\end{array}$$}

It follows that 
$$\bigoplus_k \Gamma(W, \oh(k))^{R|_W} = \FF_2[B_1,B_2,B_1B_2(C_1 A_2 + B_1 B_2)]$$
Since the restriction of $\Lambda$ to $W$ is $\Lambda|_W = C_1A_2(C_1A_2 +  B_1B_2)$, we have established that $S \cong \FF_2[B_1, B_2, \Lambda]$.

Since $B_1$ and $B_2$ are degree 1 and $\Lambda$ is degree four, it follows that $\bar M_0^{\GIT}(\PP^1,2) \cong \PP(1,1,4)$.  The open locus $M_0 (\PP^1, 2)$ is defined by the non-vanishing of $\Lambda$ so that $M_0(\PP^1,2) = \Spec \FF_2[\frac{B_1^4}{\Lambda},\frac{B_1^3 B_2}{\Lambda},\frac{B_1^2 B_2^2}{\Lambda},\frac{B_1^3 B_2}{\Lambda},\frac{B_2^4}{\Lambda}]$ which is the cone over the Veronese embedding $\PP^1 \arr \PP^4$. 
\end{proof}

\begin{remark}
In characteristic 2, $\bar{M}_0^{\GIT}(\PP^1,2)$ is $\QQ$-factorial.  Indeed, there is a resolution of singularities $\widetilde{X} \to \bar{M}_0^{\GIT}(\PP^1,2)$ with a rational exceptional divisor.  From \cite[Theorem 3]{artin_surface_singularities}, it follows that $\bar{M}_0^{\GIT}(\PP^1,2)$ is rational.   Furthermore, \cite[Proposition 17.1]{lipman} states that every two dimensional isolated rational singularity has a finite class group.  In particular, every Weil divisor on $\bar{M}_0^{\GIT}(\PP^1,2)$  has a multiple which is Cartier.  It is unclear whether $\bar{M}_0^{\GIT}(\PP^1,2)$ has finite quotient singularities. 
\end{remark}

\begin{remark}  
The morphism
$$\PP(V \tensor \FF_2) // \PGL_2 \arr (\PP V // \PGL_2) \times \FF_2$$
is a finite universal homeomorphism induced from the inclusion of rings
$\FF_2[\Delta_1, \Delta_2, \Gamma] \hookarr \FF_2[B_1, B_2, \Lambda] $
which is not surjective but the square of every element in $\FF_2[B_1, B_2, \Lambda]$ is in the image.
 \end{remark}

 %%%%%%%%%%%%%%%%%%%%%
 \subsection{Semistable and stable locus}
 
Recall from \cite[Definition 2]{seshadri_reductivity} that if $G \arr \Spec \ZZ$ is a reductive group scheme acting on projective space $\PP(V)$ where $V$ is a free $\ZZ$-module with a dual $G$-action, a geometric point $x: \Spec k \arr \PP(V)$ is \emph{semistable} if there exists a non-zero homogeneous invariant polynomial $f \in \Sym^* (V^{\dual} \tensor k)$ such that $f(x) \neq 0$.  A geometric point $x: \Spec k \arr \PP(V)$ is \emph{properly stable} if it is semistable and the $G \tensor k$-orbit is closed and the dimension of the stabilizer, $\dim G_x$, is zero.  

It follows from the explicit computation of invariants in Theorem \ref{main_theorem} that:

 \begin{cor} A geometric point $x: \Spec k \arr \PP V$ corresponding to $(s_1,s_2):\PP^1 \arr \PP^1$ is not semistable if and only if $s_1=s_2 = (\alpha x + \beta y)^2$ for some $\alpha,\beta \in k$.  The geometric point $x$ is properly stable if and only if $(s_1,s_2)$ is basepoint free. \epf
 \end{cor}
 
 For an algebraically closed field $k$, there is a unique closed point in the strictly semistable locus $ \PP (V \tensor k)^{\ss} \setminus \PP(V \tensor k)^{\s}$ corresponding to $(xy,xy)$.  The stabilizer is $\begin{pmatrix} a & 0 \\ 0 & d \end{pmatrix} \cup \begin{pmatrix} 0 & b \\ c & 0 \end{pmatrix}$.  Any other point in the strictly semistable locus is equivalent to $(xy,x(x+y))$ which has a $\ZZ_2$-stabilizer.
 
 \begin{remark} One can also prove the corollary as an easy application of the Hilbert-Mumford criterion (\cite[Theorem 2.1]{git}).  
 \end{remark}

%%%%%%%%%%%%%%%%%%%%%%%%%%%%%%%%%%%%%%%%%%%%%%%%%%
\bibliography{../../references}{}

\def\polhk#1{\setbox0=\hbox{#1}{\ooalign{\hidewidth
  \lower1.5ex\hbox{`}\hidewidth\crcr\unhbox0}}} \def\cprime{$'$}
\providecommand{\bysame}{\leavevmode\hbox to3em{\hrulefill}\thinspace}
\providecommand{\MR}{\relax\ifhmode\unskip\space\fi MR }
\providecommand{\MRhref}[2]{%
  \href{http://www.ams.org/mathscinet-getitem?mr=#1}{#2}
}
\providecommand{\href}[2]{#2}
\begin{thebibliography}{HMSV09}

\bibitem[AO01]{abramovich-oort}
Dan Abramovich and Frans Oort, \emph{Stable maps and {H}urwitz schemes in mixed
  characteristics}, Advances in algebraic geometry motivated by physics
  ({L}owell, {MA}, 2000), Contemp. Math., vol. 276, Amer. Math. Soc.,
  Providence, RI, 2001, pp.~89--100.

\bibitem[Art66]{artin_surface_singularities}
Michael Artin, \emph{On isolated rational singularities of surfaces}, Amer. J.
  Math. \textbf{88} (1966), 129--136.

\bibitem[Bor03]{borcea}
Ciprian~S. Borcea, \emph{Association for flag configurations}, Commutative
  algebra, singularities and computer algebra (Sinaia, 2002), NATO Sci. Ser. II
  Math. Phys. Chem., vol. 115, Kluwer Acad. Publ., Dordrecht, 2003, pp.~1--8.

\bibitem[FP97]{fulton-pandharipande}
W.~Fulton and R.~Pandharipande, \emph{Notes on stable maps and quantum
  cohomology}, Algebraic geometry---Santa Cruz 1995, Proc. Sympos. Pure Math.,
  vol.~62, Amer. Math. Soc., Providence, RI, 1997, pp.~45--96.

\bibitem[Gal56]{gale}
David Gale, \emph{Neighboring vertices on a convex polyhedron}, Linear
  inequalities and related system, Annals of Mathematics Studies, no. 38,
  Princeton University Press, Princeton, N.J., 1956, pp.~255--263.

\bibitem[GM82]{gelfand-macpherson}
I.~M. Gel{\cprime}fand and R.~D. MacPherson, \emph{Geometry in {G}rassmannians
  and a generalization of the dilogarithm}, Adv. in Math. \textbf{44} (1982),
  no.~3, 279--312.

\bibitem[HMSV09]{hmsv}
Benjamin Howard, John Millson, Andrew Snowden, and Ravi Vakil, \emph{The
  equations for the moduli space of {$n$} points on the line}, Duke Math. J.
  \textbf{146} (2009), no.~2, 175--226.

\bibitem[Hu05]{hu}
Yi~Hu, \emph{Stable configurations of linear subspaces and quotient coherent
  sheaves}, Q. J. Pure Appl. Math. \textbf{1} (2005), no.~1, 127--164.

\bibitem[KM97]{keel-mori}
Se{\'a}n Keel and Shigefumi Mori, \emph{Quotients by groupoids}, Ann. of Math.
  (2) \textbf{145} (1997), no.~1, 193--213.

\bibitem[Kon95]{kontsevich}
Maxim Kontsevich, \emph{Enumeration of rational curves via torus actions}, The
  moduli space of curves ({T}exel {I}sland, 1994), Progr. Math., vol. 129,
  Birkh\"auser Boston, Boston, MA, 1995, pp.~335--368.

\bibitem[Lip69]{lipman}
Joseph Lipman, \emph{Rational singularities, with applications to algebraic
  surfaces and unique factorization}, Inst. Hautes \'Etudes Sci. Publ. Math.
  (1969), no.~36, 195--279. \MR{MR0276239 (43 \#1986)}

\bibitem[LMB00]{lmb}
G{\'e}rard Laumon and Laurent Moret-Bailly, \emph{Champs alg\'ebriques},
  Ergebnisse der Mathematik und ihrer Grenzgebiete. 3. Folge. A Series of
  Modern Surveys in Mathematics [Results in Mathematics and Related Areas. 3rd
  Series. A Series of Modern Surveys in Mathematics], vol.~39, Springer-Verlag,
  Berlin, 2000.

\bibitem[Mum65]{git}
David Mumford, \emph{Geometric invariant theory}, Ergebnisse der Mathematik und
  ihrer Grenzgebiete, Neue Folge, Band 34, Springer-Verlag, Berlin, 1965.

\bibitem[Ses77]{seshadri_reductivity}
C.~S. Seshadri, \emph{Geometric reductivity over arbitrary base}, Advances in
  Math. \textbf{26} (1977), no.~3, 225--274.

\end{thebibliography}
\bibliographystyle{amsalpha}
\end{document}